\documentclass{amsart}

\usepackage[]{hyperref}
\usepackage{color}
\usepackage{pgf,tikz, tikzscale}
\usetikzlibrary{arrows, knots, external, positioning}
\newtheorem{thm}{Theorem}[section]

\newtheorem{cor}[thm]{Corollary}

\theoremstyle{definition}

\theoremstyle{remark}
\newtheorem{rem}[thm]{Remark}

\numberwithin{equation}{section}

\begin{document}

\title[On dissolving knot surgery \(4\)-manifolds under a \(\mathbb{CP}^2\)-connected sum]{On dissolving knot surgery \(4\)-manifolds under a \(\mathbb{CP}^2\)-connected sum }

\author{Hakho Choi}
\address{Department of Mathematical Sciences, Seoul National University, Seoul 08826, Korea }
\email{hako85@snu.ac.kr}

\author{Jongil Park}
\address{Department of Mathematical Sciences, Seoul National University, Seoul 08826, Korea \& Korea Institute for Advanced Study, Seoul 02455, Korea}
\email{jipark@snu.ac.kr}

\author{ Ki-Heon Yun}
\address{Department of Mathematics,  Sungshin Women's University, Seoul 02844, 
Korea}
\email{kyun@sungshin.ac.kr}%

\thanks{}
\date{\today}

\begin{abstract} 
In this article we prove that, if $X$ is a smooth $4$-manifold containing an embedded double node neighborhood, all knot surgery 4-manifolds \(X_K\) are mutually diffeomorphic to each other after a connected sum with \(\mathbb{CP}^2\). 
Hence, by applying  to the simply connected elliptic surface $E(n)$, we also show that every knot surgery 
\(4\)-manifold \(E(n)_K\)  is almost completely decomposable.
\end{abstract}

\maketitle

\section{introduction}

Since gauge theory was introduced in 1982,
 topologists and geometers working on $4$-manifolds have developed
 various techniques and they have obtained many fruitful and remarkable results
 on $4$-manifolds in last 30 years. Among them, a knot-surgery technique introduced by
 R. Fintushel and R. Stern turned out to be one of most effective techniques to modify
 smooth structures without changing the topological type of a given
 $4$-manifold~\cite{FS:98}.
 Note that Fintushel-Stern's knot surgery $4$-manifold $X_K$ is
 following:
 Suppose that $X$ is a smooth $4$-manifold containing
 an embedded torus $T$ of square $0$.
 Then, for any knot $K \subset S^3$,
 one can construct a new smooth $4$-manifold, called {\em a knot surgery $4$-manifold},
 \begin{equation*}
  X_K = X\sharp_{T=T_m} (S^1\times M_K)
 \end{equation*}
 by taking a fiber sum along a torus $T$ in $X$ and $T_m = S^1\times m$ in
 \mbox{$S^1 \times M_K$}, where $M_K$ is a $3$-manifold obtained by doing
 $0$-framed surgery along $K$ and $m$ is the meridian of $K$.
 Then Fintushel and Stern proved that, under a mild condition on $X$
 and $T$, the knot surgery $4$-manifold $X_K$ is homeomorphic,
 but not diffeomorphic, to a given $X$.
 Furthermore,
 they initially conjectured for the simply connected
 elliptic surface $E(2)$ that the classification
 of all knot surgery $4$-manifolds of the form $E(2)_K$ up to diffeomorphism
 is the same as the classification of all knots in $S^3$
 up to knot equivalence~\cite{FS_ICM98}.
 Although some partial progresses related to the conjecture were obtained by
 S. Akbulut~\cite{Akbulut:02} and M. Akaho~\cite{Akaho:06}, the full conjecture is still remained open. That is, it is not settled down  yet whether  the smooth classification of knot surgery $4$-manifolds is equivalent to the classification of prime knots in $S^3$ up to  mirror image.

On the other hand, C.T.C. Wall~\cite{Wall:1964} proved a stabilization property of smooth $4$-manifolds: 
If two simply connected smooth $4$-manifolds $X$ and $X'$ have isomorphic intersection forms, then there exists an integer $k$ such that $X \sharp k(S^2 \times S^2)$ and  $X' \sharp k(S^2 \times S^2)$ are diffeomorphic to each other. One of the interesting questions on the stabilization problem is to find the smallest such an integer $k$. 
S. Akbulut~\cite{Akbulut:02} and D. Auckly~\cite{Auckly:00} showed that $k=1$ is enough for a family of knot surgery $4$-manifolds $X_K$. That is,  
\(X_K \sharp (S^2 \times S^2)\) is diffeomorphic to \(X \sharp (S^2 \times S^2) \)  and  
\(X_K \sharp (S^2 \tilde\times S^2)\) is also diffeomorphic to \(X \sharp (S^2 \tilde\times S^2)\) for any knot $K$ in $S^3$.
R. Baykur and N. Sunukjian~\cite{BaySun:2013} also proved a single stabilization for a family of $4$-manifolds obtained by logarithmic transforms.  

In the same spirit as stabilization problems, topologists have also studied whether two smooth $4$-manifolds with isomorphic intersection forms are diffeomorphic to each other after a connected sum with \(\mathbb{CP}^2\).
In this article, we obtain an affirmative answer to this problem for a large family of knot surgery $4$-manifolds. 
That is, we prove that, if $X$ is a smooth $4$-manifold containing an embedded double node neighborhood,  a codimension zero submanifold obtained from $(S^1 \times S^1)\times D^2$ by attaching two $(-1)$-framed $2$-handles along the first $S^1$ factor in $\partial ((S^1 \times S^1)\times D^2)$, then all knot surgery 4-manifolds \(X_K\) obtained by performing a knot surgery operation along a torus in the double node neighborhood become mutually diffeomorphic after a connected sum with  \(\mathbb{CP}^2\). Explicitly, we get 

\begin{thm}\label{Thm-1}
Suppose that $X$ is a smooth $4$-manifold containing an embedded double node neighborhood. Let  \(X_K\) be a knot surgery $4$-manifold obtained by 
performing a knot surgery operation along a torus in the double node neighborhood. Then
\(
X_K \sharp \mathbb{CP}^2 \text{ is diffeomorphic to }  X \sharp \mathbb{CP}^2
\)
 for any knot $K$ in $S^3$.
\end{thm}

\begin{rem}
Recently S. Akbulut suggested to us that Theorem 1.1 above is still true for a smooth $4$-manifold containing a fishtail neighborhood. The essential part is that, by using the $(+1)$-framed $2$-handle coming from $\mathbb{CP}^2$ blow-up and a $(-1)$-framed $2$-handle in a fishtail neighborhood, we can get a $0$-framed $2$-handle in a meridian of the slice $1$-handle which is linked with a $(+1)$-framed 2-handle. We slide this $(+1)$-framed $2$-handle over the slice $1$-handle and then cancel the slice $1$-handle and a $0$-framed $2$-handle pair. By this way the slice $1$-handle turns into $(+1)$-framed 2-handle, then by the $2$-handle slides indicated in ~\cite[Figure 6.13 of p.72]{Akbulut:2016book}, we can dissolve the manifold. 
\end{rem}

Finally, people also studied  an \emph{almost completely decomposable} (ACD) property for simply connected smooth $4$-manifolds. Note that a simply connected smooth $4$-manifold $X$ is \emph{completely decomposable} if $X$ is diffeomorphic to $\sharp k\mathbb{CP}^2 \sharp \ell \overline{\mathbb{CP}^2}$ for some integers $k$ and $\ell$, and \emph{almost completely decomposable} if $X\sharp \mathbb{CP}^2$ is completely decomposable.  
R. Mandelbaum and B. Moishezon showed that many complex surfaces are almost completely decomposable. For example, they showed that smooth hypersurfaces in $\mathbb{CP}^3$, simply connected elliptic surfaces $E(n)$ and $E(n)_{p,q}$  are almost completely decomposable~\cite{Mandelbaum:1976, Moishezon:77, MANDELBAUM:1980aa, Gompf:1989}.
In this article, we also investigate this problem for knot surgery $4$-manifolds $E(n)_K$, where $E(n)$ is the simply connected elliptic surface with Euler characteristic $12n$.  
Since it is well known that the Dolgachev surface $E(1)_{2,3}$ can be identified with a knot surgery $4$-manifold $E(1)_K$, where $K$ is the trefoil knot, it is natural to ask whether every knot surgery $4$-manifold $E(n)_K$ is almost completely decomposable or not. 
By applying Theorem~\ref{Thm-1} above  to  $E(n)$ and  combining a Moishezon's old result~\cite{Moishezon:77} that every simply connected elliptic surface is almost completely decomposable, we conclude that  

\begin{cor} \label{Cor-1}
For any knot $K$ in $S^3$, 
\( E(n)_K \sharp \mathbb{CP}^2\) is completely decomposable.
\end{cor}

\begin{rem}
Shortly after this article was announced, R. Baykur gave an alternative proof of Corollary~\ref{Cor-1} using $5$-dimensional cobordism arguments~\cite{Baykur:2017}. 
\end{rem}


\subsection*{Acknowledgements}

The authors would like to thank Selman Akbulut, Ronald Fintushel, Robert Gompf, Daniel Ruberman and Ronald Stern for their valuable comments. Jongil Park was supported by Samsung Science and Technology Foundation under Project Number SSTF-BA1602-02. 
He also holds a joint appointment at KIAS and in the Research Institute of Mathematics, SNU. Ki-Heon Yun was partially supported by the Basic Science Research Program through the National Research Foundation of Korea  funded by the Ministry of Education (2015R1D1A1A01058941).



\section{A Kirby diagram of knot surgery \texorpdfstring{$4$}{4}-manifolds}
\label{Section-2}

In this section we first briefly review how to draw a Kirby diagram of $E(K) \times S^1$, where $E(K) = S^3 - \nu(K)$ denotes the knot complement of $S^3$. And then we find some conditions on the Kirby diagram so that we can get a smooth operation in a knot surgery $4$-manifold $X_K$ which changes the knot $K$ to $K'$ by adding a full twist to $K$ but $X_{K'}$ remains diffeomorphic to $X_K$. 

\subsection{Kirby diagram of $E(K) \times S^1$}

It is well known how to draw a Kirby diagram of $E(K) \times S^1$, which is following (refer to \cite{Akbulut:02, Akbulut:2016book} or \cite{GS:99}): 
%
%
\begin{figure}[htbp]
\begin{center}
\begin{tikzpicture}[scale=0.8]
\begin{scope}[scale=0.5]
\draw (2,2.5)--(4,2.5)--(4,5.5)--(2,5.5)--cycle;
\draw (2.25, 4) node[draw=none, fill=none, right] {$T_K$};
\draw[very thick] (2,3) .. controls +(-2,0) and +(-2,0) .. (2,1)-- (4,1) ..controls +(2,0) and +(2,0) .. (4,3);
\draw[very thick] (2,5) .. controls +(-2,0) and +(-2,0) .. (2,7)-- (4,7) ..controls +(2,0) and +(2,0) .. (4,5);
\draw (3,3) circle (3);
\draw (0, 3) .. controls +(0, -0.5) and +(0, -0.5) .. (6, 3);
\draw[dashed] (0, 3) .. controls +(0, 0.5) and +(0, 0.5) .. (6, 3);
\draw (6, 3) node[draw=none, fill=none, right] {$B_1^3$};
\draw (-2, -0.5) node[draw=none, fill=none, right] {(a) Union of two $2$-tangles};
\end{scope}

\begin{scope}[shift={(6,0)}, scale=0.5]
\draw (2,2.5)--(4,2.5)--(4,5.5)--(2,5.5)--cycle;
\draw (2.25, 4) node[draw=none, fill=none, right] {$T_K$};
\draw[very thick, ->] (2,3) .. controls +(-2,0) and +(-2,0) .. (2,1)-- (4,1) ..controls +(2,0) and +(2,0) .. (4,3);
\draw[very thick, <-] (2,5) .. controls +(-2,0) and +(-2,0) .. (2,7)-- (4,7) ..controls +(2,0) and +(2,0) .. (4,5);
\draw (3,4) circle (2.5);
\draw (0.5, 4) .. controls +(0, -0.5) and +(0, -0.5) .. (5.5, 4);
\draw[dashed] (0.5, 4) .. controls +(0, 0.5) and +(0, 0.5) .. (5.5, 4);
\draw (5.5, 4) node[draw=none, fill=none, right] {$B_1^3$};
\draw (-2, -0.5) node[draw=none, fill=none, right] {(b) Union of two $4$-tangles};
\end{scope}
\end{tikzpicture}
\caption{Decomposition of $S^3$ and $K$}
\label{Fig-1}
\end{center}
\end{figure}
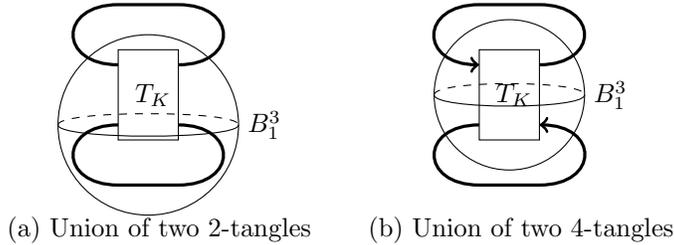
%
Let us consider $S^3$ as a union of two $3$-balls $B_1^3 \cup B_2^3$, $B_1^3 \cap B_2^3 = \partial B_1^3 = \partial B_2^3$, as in Figure~\ref{Fig-1} (a), so that $K_i = K \cap B_i^3$ is a $2$-tangle for each $i=1, 2$. 
Then $E(K) \times S^1$ is obtained from $E(K) \times [0,1]$ by identifying one end $E(K) \times \{0\}$ with the other end $E(K) \times \{1\}$ and $E(K)= (B_1^3 - \nu(K_1) )\cup (B_2^3 - \nu(K_2) )$. 
Since $(B_1^3 - \nu(K_1) ) \times [0,1]$ is the  complement of an embedded disk from the $4$-ball $B^4 = B_1^3 \times [0,1]$, we can draw it as a carving $1$-handle which is dotted $K\sharp K^*$ as in Figure~\ref{Fig-2}, where $K^*$ is the mirror image of $K$. We will get one $1$-handle, the dotted circle at the bottom of Figure~\ref{Fig-2}, and one $0$-framed $2$-handle per each $1$-handle in $(B_1^3 - \nu(K_1))$ drawn as a $0$-framed $2$-handle going through the dotted circle at the bottom of Figure~\ref{Fig-2}.

\begin{figure}[htbp]
\begin{center}

\begin{tikzpicture}[scale= 0.5]
\begin{scope}
\begin{knot}[
	clip width=5,
	clip radius = 2pt,
	end tolerance = 1pt,
]
\strand[very thick] (4, 5)--(8,5);
\strand[thick] (2,3.5)--(4,3.5)--(4,5.5)--(2,5.5)--cycle;
\strand[thick, shift={(6,0)}] (2,3.5)--(4,3.5)--(4,5.5)--(2,5.5)--cycle;
\strand[black, very thick]  (10,5) .. controls +(2,0) and +(2,0) .. (10,7) -- (2,7)..controls +(-2,0) and +(-2,0) .. (2,5);
\strand[black, very thick] (2,4)..controls +(-1,0) and +(-1,0) ..(2,3)--(4,3)..controls +(1,0) and +(1,0)..(4,4);
\strand[black, very thick, shift={(6,0)}] (2,4)..controls +(-1,0) and +(-1,0) ..(2,3)--(4,3)..controls +(1,0) and +(1,0)..(4,4);
\strand (2,3.8)..controls +(-0.8,0) and +(-0.8,0) ..(2,3.2)--(4,3.2)..controls +(0.8,0) and +(0.8,0)..(4,3.8);
\strand (4,5.2)--(5,5.2)..controls +(1,0) and +(1,0)..(5,6.8)-- (2,6.8)..controls +(-1.8,0) and +(-1.8,0) .. (2,5.2);
\strand[blue, very thick, shift={(0,0.3)}] (5, -0.5)..controls +(0, -0.5) and +(0, -0.5)..(6, -0.5)--(6, 1.5)..controls +(0,0.5) and +(0,0.5) ..(5,1.5)--(5,-0.5);
\strand[black, thick] ((2,1.5)..controls +(0,-0.5) ..(3,1)--(9,1) ..controls +(1,0) ..(10,3)..controls +(0, 0.5) and +(0,0.5)..(9,3)--(9,2)..controls +(0, -0.5)..(8,1.5)--(4,1.5)..controls +(-1,0) ..(3,3)..controls +(0, 0.5) and +(0,0.5)..(2,3)--cycle;
\strand[black, thick] (0.5,0)--(11.5,0)..controls +(0.5, 0) ..(12, 6)..controls +(0,0.5) and +(0,0.5) ..(11,6)--(11,1)..controls +(0,-0.5) ..(10.5, 0.5)--(1.5, 0.5)..controls +(-0.5, 0)..(1, 1)--(1, 6)..controls +(0,0.5) and +(0,0.5) .. (0,6)--(0, 0.5)..controls +(0,-0.5)..(0.5,0);
\strand (5.5,2.2) ellipse (0.3 and 0.4);
\strand (8,7) ellipse (0.4 and 0.6);
\flipcrossings {1,4,5, 7,9, 12,14,15,16,19,20,23}
\end{knot}

\draw (5, 0.75) node[circle, fill, inner sep=2pt, blue] {}; 
\draw (6, 5) node[circle, fill, inner sep=2pt, black] {}; 
\draw (12, 0) node[right] {$0$}; 
\draw (10, 1) node[right] {$0$}; 
\draw (1.2, 4.6) node[right] {$\vdots$}; 
\draw (10.2, 4.6) node[right] {$\vdots$}; 
\draw (2.5, 4.5) node[draw=none, fill=none, right] {\Large $T_K$};
\draw (8.5, 4.5) node[draw=none, fill=none, right] {\Large $T_K^*$};
\draw (5.5, 3) node {$b$}; 
\draw (8, 6) node {$a$}; 
\draw (6, 6) node {$c$}; 
\draw (3, 6.9) node[shape=rectangle, draw=black, fill=white, fill opacity=1] {$\ell$}; 
\end{scope}

\draw[very thick, ->] (12.5, 3)-- node[above, text width=2cm, align=center] {$\phi$} (13.5,3);

\begin{scope}[shift={(15,0)}]
\begin{knot}[
	clip width=4,
	end tolerance = 1pt,
]
\strand[thick] (0,0)--(4,0) ..controls +(1,0).. (5,1)--(5,5) ..controls +(0,1).. (4,6)--(0,6)..  controls +(-1,0) and +(-1,0).. (0,4.5)--(2,4.5) ..controls +(1,0).. (3,3.5)--(3,2.5)..controls +(0,-1)..(2,1.5)--(0,1.5)..  controls +(-1,0) and +(-1,0).. (0,0);
\strand[thick] (4,3) ellipse (1.5 and 0.5);
\strand[thick] (0.4,3.1) ellipse (0.5 and 2.3);
\strand (0.1,3.1) ellipse (0.6 and 0.5);
\strand (4.2,2.3) ellipse (0.5 and 0.6);
\strand (3,6) ellipse (0.4 and 0.5);
\flipcrossings {1,3,6,7,9,12,14}
\end{knot}
\draw (3.5, 2) node {$b$}; 
\draw (3.7, 6.5) node {$a$}; 
\draw (-0.8, 3) node {$c$}; 
\draw (-0.5, 4) node {$0$}; 
\draw (1.5, 4.5) node[circle, fill, inner sep=2pt, black] {}; 
\draw (2.5, 3) node[circle, fill, inner sep=2pt, black] {}; 
\end{scope}
\useasboundingbox ([shift={(1mm,1mm)}]current bounding box.north east) rectangle ([shift={(-1mm,-1mm)}]current bounding box.south west);
\end{tikzpicture}

\caption{A Kirby diagram of $E(K) \times S^1$ and a boundary diffomorphism}
\label{Fig-2}
\end{center}
\end{figure}
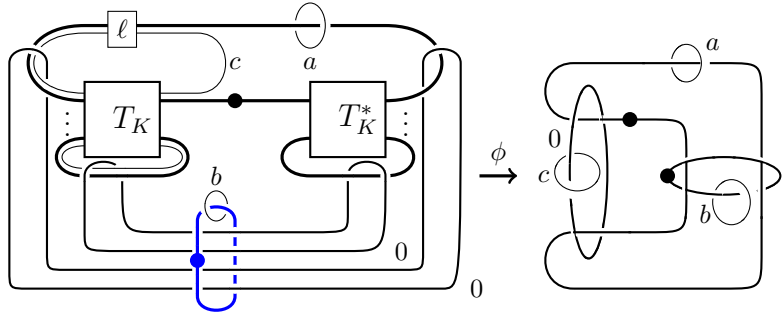

Since the knot surgery $4$-manifold $X_K$ is constructed as a union $(X - \nu(F)) \cup_\phi (E(K) \times S^1)$, where $F$ is an embedded torus of square $0$ in $X$ and a diffeomorphism $\phi : \partial (E(K) \times S^1) \to \partial (X - \nu(F)) = \partial (F \times D^2)$ is chosen so that $\phi([m_K \times S^1]) = [F]$ and $\phi([\ell_K])= [\partial D^2]$,  we have to identify the images of three simple closed curves named $a$, $b$ and $c$ under the map $\phi$. 
It is well known that the map $\phi$ sends $a$, $b$ and $c$ in the left figure to the same lettered circle in the right figure as in Figure~\ref{Fig-2} ~\cite{Akbulut:02} respectively. Here $\ell$ denotes the negative of the blackboard framing of $K$ and the small box with letter $\ell$ means $\vert \ell \vert$ times right-handed full twists if $\ell$ is a positive integer and left-handed full twists if $\ell$ is a negative integer. 
In this article we always assume that the curve $c$ in left side of Figure~\ref{Fig-2}, which is corresponding to a $0$-framed longitude of $K$, is given by blackboard framing in the $4$-tangle part drawn as $T_K$ in a rectangle.

\subsection{Smooth operation on  $X_K$}
Recently R. Gompf~\cite{Gompf:2016a, Gompf:2016} constructed an infinite order cork, which is a contractible $4$-manifold $C$ with an infinite-order self-diffeomorphism of its boundary $f: \partial C  \to \partial C$ not extending to a diffeomorphism of $C$. S. Akbulut~\cite{Akbulut:2016} and M. Tange~\cite{Tange:2016} also constructed an example of infinite order corks.
  It is known that the existence of cork is closely related to that of exotic smooth structures on $4$-manifolds and the construction of an infinite order cork is based on how a knot surgery $4$-manifold is changed under torus twist or $\delta$-move. Gompf also found a condition on torus twist which does not change a smooth structure in some cases. 
Similarly, we try to find a smooth operation in the Kirby diagram of a knot surgery $4$-manifold which does not change a smooth structure. 
For this purpose, we first review  torus twist or $\delta$-move. 
Here is a description of Akbulut's $\delta$-move (\cite{Akbulut:2016, Gompf:2016}): 
 Let $X$ be any $4$-manifold with boundary, $\gamma$ be a circle in $\partial X$, and $\delta \subset X$ be an unknot in $\partial X$ obtained by connected summing two parallel copies of $\gamma$ along a possibly  complicated band.  Then $\delta$-move is a diffeomorphism 
$ f_\delta : \partial X \to \partial X $
obtained by first introducing a $2$-handle/$3$-handle canceling pair whose $2$-handle is attached along $\delta$ with $0$-framing, 
then blowing up along $\gamma_\pm$ a $(\pm1)$-framed circle, sliding it along the $0$-framed $\delta$, and then blowing down along $\gamma_\mp$ circle again. 
This procedure is explained in Figure~\ref{Fig-3} below.
 

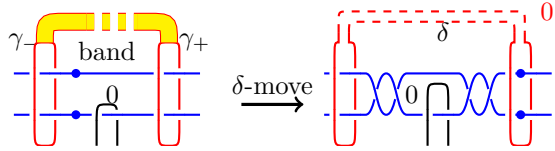
\begin{figure}[htbp]

\begin{tikzpicture}[scale=0.55]
\begin{knot}[
	scale=0.5,
	clip width=5,
	clip radius = 2pt,
	end tolerance = 0.5pt,
	consider self intersections=true,
]

\strand[red, thick] (0.5,0) .. controls +(0.5,0) .. (1,1) -- (1, 4.5) .. controls +(0,0.5)  .. (0.5,5) .. controls  +(-0.5,0)  .. (0,4.5) --(0,1) .. controls +(0,-1) ..(0.5,0);
\strand[red, thick, shift={(6,0)}] (0.5,0) .. controls +(0.5,0) .. (1,1) -- (1, 4.5) .. controls +(0,0.5)  .. (0.5,5) .. controls  +(-0.5,0)  .. (0,4.5) --(0,1) .. controls +(0,-1) ..(0.5,0);
\strand[blue, thick] (-1,1.5) -- (8, 1.5);
\strand[blue, thick] (-1, 3.5) -- (8, 3.5);
\strand[thick] (3,0)--(3, 1) .. controls +(0,1) ..(3.5,2) .. controls +(0.5,0) .. (4,1) -- (4,0);
\strand[red, double=yellow, double distance=6pt] (0.5,5) .. controls +(0,1) .. (1.5, 6)--(2.5,6);
\strand[red, double=yellow, double distance=6pt, dashed] (2.5,6)--(5,6);
\strand[red, double=yellow, double distance=6pt] (5,6)--(5.5,6)..controls +(1,0) .. (6.5, 5);
\draw (-0.5, 5) node[draw=none, fill=none] {$\gamma_-$};
\draw (7.75, 5) node[draw=none, fill=none] {$\gamma_+$};
\draw (3, 2.5) node[draw=none, fill=none, right] {$0$};
\draw (3.5, 5.5) node[draw=none, fill=none, below] {band};
\draw (2, 1.5) node[circle, fill, inner sep=1.2pt, blue] {}; 
\draw (2, 3.5) node[circle, fill, inner sep=1.2pt, blue] {}; 
\flipcrossings {1,2,5,6,9};
\end{knot}

\draw[very thick, ->] (5, 1)-- node[above, text width=2cm, align=center] {$\delta$-move} (6.5,1);

\begin{scope}[shift={(8, 0)}]
\begin{knot}[
	scale=0.5,
	clip width=5,
	clip radius = 2pt,
	end tolerance = 0.5pt,
	consider self intersections=true,
]
\strand[blue, thick] (-2,1.5)--(0,1.5) .. controls +(0.5,0) and +(-0.5,0) .. (1, 3.5) .. controls +(0.5,0) and +(-0.5,0)..(2, 1.5)--(4.5, 1.5) ..controls +(0.5,0) and +(-0.5,0) .. (5.5, 3.5) .. controls +(0.5,0) and +(-0.5,0)..(6.5, 1.5)-- (9, 1.5);
\strand[blue, thick] (-2,3.5)--(0,3.5) .. controls +(0.5,0) and +(-0.5,0) .. (1, 1.5) .. controls +(0.5,0) and +(-0.5,0)..(2, 3.5)--(4.5, 3.5) ..controls +(0.5,0) and +(-0.5,0) .. (5.5, 1.5) .. controls +(0.5,0) and +(-0.5,0)..(6.5, 3.5)-- (9, 3.5);
\strand[black, thick] (3,0) --(3,2)..controls +(0,1) .. (3.5, 3)..controls +(0.5,0).. (4,2) --(4,0);
\strand[red, thick,shift={(-1.5, 0)}] (0.5,0) .. controls +(0.5,0) .. (1,1) -- (1, 4) .. controls +(0,1)  .. (0.7,5)--(0.7,5.5)  (0.2,5.5)--(0.2,5) .. controls  +(-0.2,0)  .. (0,4) --(0,1) .. controls +(0,-1) ..(0.5,0);
\strand[red, thick, shift={(7,0)}] (0.5,0) .. controls +(0.5,0) .. (1,1) -- (1, 4) .. controls +(0,1)  .. (0.7,5)--(0.7,5.5)  (0.2,5.5)--(0.2,5) .. controls  +(-0.2,0)  .. (0,4) --(0,1) .. controls +(0,-1) ..(0.5,0);
\strand[red, thick, dashed] (-0.8,5.5) .. controls +(0,0.5) .. (-0.3,6)--(6.7,6) .. controls +(0.5,0) .. (7.2, 5.5);
\strand[red, thick, dashed] (-1.3,5.5)--(-1.3,6) .. controls +(0,0.5) .. (-0.8,6.5)--(7.2,6.5) .. controls +(0.5,0) .. (7.7, 6)--(7.7,5.5);
\draw (3, 2.5) node[draw=none, fill=none, left] {$0$};
\draw (7.5, 1.5) node[circle, fill, inner sep=1.2pt, blue] {}; 
\draw (7.5, 3.5) node[circle, fill, inner sep=1.2pt, blue] {}; 
\draw (3, 5.5) node[draw=none, fill=none, right] {$\delta$};
\draw (8, 6.5) node[draw=none, fill=none, right, red] {$0$};
\flipcrossings {2, 12, 3, 5, 8, 10, 14};
\end{knot}
\end{scope}

\end{tikzpicture}

\caption{$\delta$-move}
\label{Fig-3}
\end{figure}

We note that $\delta$-move is a boundary diffeomorphism and  this diffeomorphism usually does not extend to the whole $4$-manifold $X$.
The main reason is that blowing-up and blowing-down operations along $\gamma_\pm$ in $\delta$-move are boundary operation, not a $4$-dimensional operation.
But in some case the boundary diffeomorphism can extend to the whole $4$-manifold and such a phenomenon was already observed by Gompf~\cite[Section 4]{Gompf:2016}. 
In a Kirby diagram of knot surgery $4$-manifolds, when we perform Kirby moves, it is not easy to control a carving $1$-handle, coming from carving out a slice disk. 
But we can observe that, if we have a $(\pm 1)$-framed $2$-handle along $\gamma_-$, two strands of a $1$-handle (or two $1$-handles) going through $\gamma_-$ and a $0$-framed $2$-handle located as in Figure~\ref{Fig-3}, then there is a method to introduce a pair of opposite twists on these two strands:  If we slide twice this $0$-framed $2$-handle in Figure~\ref{Fig-3} over  $(\pm 1)$-framed $2$-handle parallel to $\gamma_-$, then it gives a pair of twists on the two strands passing through the circle $\gamma$. So we can generate a pair of opposite twists between two strands, a part of $1$-handles, without blowing-up/blowing-down operations.
This operation sends  the $(\pm 1)$-framed $2$-handle along $\gamma_-$ to the $(\pm 1)$-framed $2$-handle along $\gamma_+$ (Figure~\ref{Fig-7}).
Now assume that $K'$ is a knot obtained from $K$ by adding a full twist corresponding to the above operation.
Then we have to find a sequence of handle slides which sends the $(\pm 1)$-framed $2$-handle along $\gamma_+$ back to the $(\pm 1)$-framed $2$-handle along $\gamma_-$ and also sends a $0$-framed longitude of $K$ to a $0$-framed longitude of $K'$ in the Kirby diagram because we needs an operation without changing its smooth structure.
Under the conditions that 
\begin{itemize}
\item[(1)] there is a $(\pm 1)$-framed $2$-handle along $\gamma_-$ and a $0$-framed $2$-handle as in Figure~\ref{Fig-3} so that Kirby moves in Figure~\ref{Fig-7} can be applied,
\item[(2)] there is a sequence of Kirby moves which sends the $(\pm 1)$-framed $2$-handle along $\gamma_+$ to the $(\pm 1)$-framed $2$-handle along $\gamma_-$ and
\item[(3)] the two strands are oriented oppositely when we consider $K$ as an oriented knot so that this process sends a $0$-framed longitude of $K$ to that of $K'$,
 \end{itemize}
we can get a smooth operation in a knot surgery $4$-manifold $X_K$ which changes the knot $K$ to $K'$ by adding a full twist to $K$ but $X_{K'}$ remains diffeomorphic to $X_K$. 
We will prove in Section~\ref{Section-3} that,  if we take a connected sum of $\mathbb{CP}^2$ with a smooth $4$-manifold $X$ which contains an embedded double node neighborhood and if we perform a knot surgery operation in the double node neighborhood, then these conditions are satisfied.


%
\section{Proof of Theorem~\ref{Thm-1}}\label{Section-3}
%
Let $X$ be a smooth $4$-manifold which contains an embedded double node neighborhood, a codimension zero submanifold obtained from $(S^1 \times S^1)\times D^2$ by attaching two $(-1)$-framed $2$-handles along  along the first $S^1$ factor in $\partial ((S^1 \times S^1)\times D^2)$,  and $K$ be any knot in $S^3$. We perform a knot surgery operation along this torus and let $X_K$ be the resulting knot surgery $4$-manifold.

First we will show that $X_K \sharp \mathbb{CP}^2$ is diffeomorphic to $X_{K'} \sharp \mathbb{CP}^2$ using a Kirby diagram of $E(K) \times S^1$ explained in Section~\ref{Section-2}, where $K'$ is a knot obtained by changing a crossing in $K$. 
If we perform a knot surgery along the torus of square $0$ in the embedded double node neighborhood,
then the two vanishing cycles, $(-1)$-framed $2$-handles, can be drawn as in Figure~\ref{Fig-4}.
Note that Figure~\ref{Fig-4} without the two $(-1)$-framed $2$-handles is diffeomorphic to Figure~\ref{Fig-2}  because they are related by a sequence of $1$-handle slides over $1$-handle and  $1$-handle/$2$-handle pair cancellations (refer to ribbon move~\cite[Chapter 6]{GS:99}). 
%
%
\begin{figure}[hbtp]
\begin{center}

\begin{tikzpicture}[]
\begin{knot}[
	scale=0.33,
	clip width=5,
	clip radius = 2pt,
	end tolerance = 0.5pt,
	consider self intersections=true,
]

\strand (1,1)--(25,1);
\strand (1,3)--(2,3);
\strand (4,3)--++(18,0);
\strand (24,3)--++(1,0);
\strand (1,5)--(2,5);
\strand (4,5)--++(18,0);
\strand (24,5)--++(1,0);
\strand (1,7)--++(1,0);
\strand (4,7)--(25,7);
\strand (1,1)..controls +(-1,0) and +(-1,0)..++(0,2);
\strand (1,5)..controls +(-1,0) and +(-1,0)..++(0,2);
\strand (25,1)..controls +(1,0) and +(1,0)..++(0,2);
\strand (25,5)..controls +(1,0) and +(1,0)..++(0,2);
\strand[blue, thin] (1,1.5)--++(5,0);
\strand[blue, thin] (1,2.5)--++(1,0);
\strand[blue, thin] (4,2.5)--++(2,0);
\strand[blue, thin] (1,5.5)--++(1,0);
\strand[blue, thin] (4,5.5)--++(2,0);
\strand[blue, thin] (1,6.5)--++(1,0);
\strand[blue, thin] (4,6.5)--++(2,0);
\strand[blue, thin] (1,5.5)..controls +(-0.5, 0) and +(-0.5,0)..++(0,1);
\strand[blue, thin] (1,1.5)..controls +(-0.5, 0) and +(-0.5,0)..++(0,1);
\strand[blue, thin] (6,1.5)..controls +(0.5, 0) and +(0.5,0)..++(0,1);
\strand[blue, thin] (6,5.5)..controls +(0.5, 0) and +(0.5,0)..++(0,1);
\strand[red] (10, 5.6)..controls +(0.5,0) and +(0.5,0) ..++(0,-1)..controls +(-0.5,0) and +(-0.5,0) ..++(0, 1);
\strand[red] (10, 3.4)..controls +(0.5,0) and +(0.5,0) ..++(0,-1)..controls +(-0.5,0) and +(-0.5,0) ..++(0, 1);
\strand (13, 3.5)..controls +(0.5,0) and +(0.5,0) ..++(0,-3)..controls +(-0.5,0) and +(-0.5,0) ..++(0, 3);
\strand (13, 0.2)..controls +(0.5,0) and +(0.5,0) ..++(0,-3)..controls +(-0.5,0) and +(-0.5,0) ..++(0, 3);

\strand (2,0)--++(0,1.4)..controls +(0, 0.5) and +(0,0.5)..++(0.5,0)--++(0, -1.4);
\strand (24,0)--++(0,1.4)..controls +(0, 0.5) and +(0,0.5)..++(0.5,0)--++(0, -1.4);
\strand (2,0)..controls +(0, -1)..++(0.5,-1)--++(21.5,0)..controls +(0.5,0)..++(0.5,1);
\strand ((2.5,0)..controls +(0, -0.5)..++(0.5,-0.5)--++(20.5,0)..controls +(0.5,0)..++(0.5,0.5);
\strand (-0.5, 6.2)--++(1.3,0)..controls +(0.5,0) and +(0.5,0)..++(0,-0.5)--++(-1.3,0);
\strand (-0.5,6.2) ..controls +(-1, 0)..++(-1,-0.7)--++(0, -7);
\strand (-0.5,5.7) ..controls +(-0.5, 0)..++(-0.5,-0.7)--++(0, -6);

\strand (26.5, 6.2)--++(-1.3,0)..controls +(-0.5,0) and +(-0.5,0)..++(0,-0.5)--++(1.3,0);
\strand (26.5,6.2) ..controls +(1, 0)..++(1,-0.7)--++(0, -7);
\strand (26.5,5.7) ..controls +(0.5, 0)..++(0.5,-0.7)--++(0, -6);

\strand (-1,-1)..controls +(0, -0.5)..++(0.5,-0.5)--++(27,0)..controls +(0.5,0)..++(0.5,0.5);
\strand (-1.5,-1)..controls +(0, -1)..++(0.5,-1)--++(28,0)..controls +(0.5,0)..++(0.5,1);

\flipcrossings {14, 20, 17, 3, 12, 8, 10, 2, 21, 22, 23,24,16, 6};

\draw (2,2.3) rectangle (4,5.7);
\draw (3,4) node {\Large $T_K$}; 
\draw (22,2.3) rectangle (24,5.7);
\draw (23,4) node {\Large $T_K^*$}; 
\draw (2,6.3) rectangle (4,7.2);
\draw (3,6.7) node {\tiny $\ell$}; 

\draw (15, 7) node[circle, fill, inner sep=1.2pt] {}; 
\draw (15, 3) node[circle, fill, inner sep=1.2pt] {}; 

\draw (13,-2.8) node[circle, fill, inner sep=1.2pt, black] {}; 
\draw (0,4.3) node[right] {$\vdots$}; 
\draw (25,4.3) node[right] {$\vdots$}; 
\draw (10, 5.6) node[above, red] {$-1$}; 
\draw (10, 2.4) node[below, red] {$-1$}; 
\draw (13,3.5) node[above, black] {$0$}; 
\draw (0,0) node[right] {$\cdots$}; 
\draw (25,0) node[right] {$\cdots$}; 
\draw (-1,0) node[right] {$0$}; 
\draw (3,0) node[right] {$0$}; 
\end{knot}
\useasboundingbox ([shift={(1mm,1mm)}]current bounding box.north east) rectangle ([shift={(-1mm,-1mm)}]current bounding box.south west);
\end{tikzpicture}

\caption{$E(K) \times S^1$ in a double node neighborhood}
\label{Fig-4}
\end{center}
\end{figure}
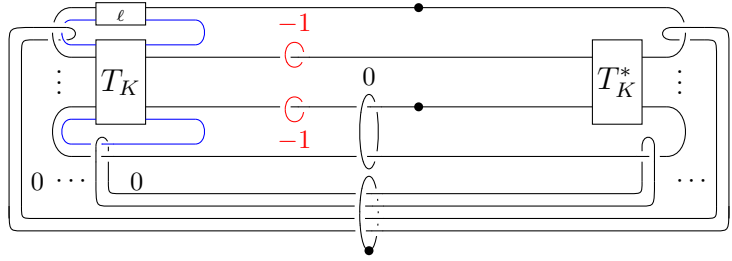

If we take a connected sum with $\mathbb{CP}^2$, then a disjoint $(+1)$-framed unknot is introduced. Now we perform a series of $2$-handle slides as in Figure~\ref{Fig-5}: First slide the $(+1)$-framed $2$-handle over two $(-1)$-framed $2$-handles coming from vanishing cycles in the double node neighborhood, so that we get a $(-1)$-framed $2$-handle as in the middle of Figure~\ref{Fig-5}. And slide again two vanishing cycles over the $(-1)$-framed $2$-handle to get a linked $0$-framed $2$-handles in the last of Figure~\ref{Fig-5}. By using this process, we get Figure~\ref{Fig-6} from Figure~\ref{Fig-4}.
Now we get Figure~\ref{Fig-8} by sliding the $0$-framed $2$-handle in the middle of Figure~\ref{Fig-6} as in Figure~\ref{Fig-7}.
And then slide  the $(-1)$-framed $2$-handle over two $0$-framed $2$-handles passing through the bottom dotted circle to get Figure~\ref{Fig-9}.
We slide twice an embedded circle corresponding to the $0$-framed longitude of the knot $K$ over the $(-1)$-framed $2$-handle in order to get Figure~\ref{Fig-10}.  Let $T_{K'}$ be a $4$-tangle obtained from $T_K$ by adding a right-handed full  twist as in Figure~\ref{Fig-11}. Then we apply a reversed sequence of Kirby moves of Figure~\ref{Fig-5}, so that we get Figure~\ref{Fig-12}.


\begin{figure}[htbp]
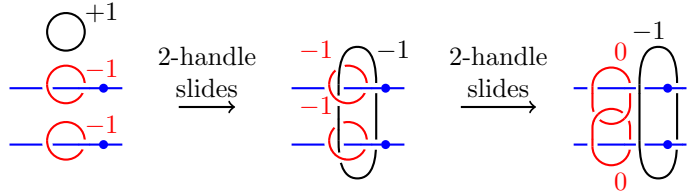

\begin{center}



\caption{Kirby moves}
\label{Fig-5}
\end{center}
\end{figure}



\begin{figure}[htbp]
\begin{center}

%

\caption{}
\label{Fig-6}
\end{center}
\end{figure}



\begin{figure}[htbp]
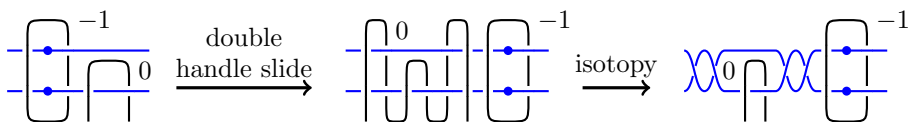


%
\caption{Double slides of a $0$-framed $2$-handle over a $(-1)$-framed $2$-handle}
\label{Fig-7}
\end{figure}

\begin{figure}[htbp]
\begin{center}
%
\caption{}
\label{Fig-8}
\end{center}
\end{figure}


\begin{figure}[htbp]
\begin{center}

%

\caption{}
\label{Fig-9}
\end{center}
\end{figure}


\begin{figure}[htbp]
\begin{center}

%

\caption{}
\label{Fig-10}
\end{center}
\end{figure}

\begin{figure}
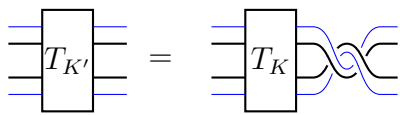

%

\caption{A longitude given by a blackboard framing in a $4$-tangle diagram}
\label{Fig-11}
\end{figure}


\begin{figure}[htbp]
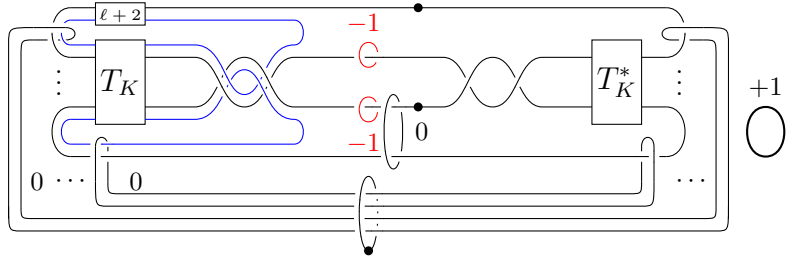

\begin{center}

%

\caption{$\big( E(K') \times S^1\big) \sharp \mathbb{CP}^2$ in a double node neighborhood}
\label{Fig-12}
\end{center}
\end{figure}

Since we select an orientation of a knot diagram of $K$ and we choose a $4$-tangle as in Figure~\ref{Fig-1} (b), 
the knot $K'$ obtained by closing the tangle $T_{K'}$ as in Figure~\ref{Fig-1} has a blackboard framing $-(\ell + 2)$, so that the embedded circle in Figure~\ref{Fig-12} is a $0$-framed longitude of $K'$, where $-\ell$ is the blackboard framing of $K$. Hence it implies that
\(
X_K \sharp \mathbb{CP}^2 
\)
is diffeomorphic to 
\(
X_{K'} \sharp \mathbb{CP}^2
\).

Note that, in the process of Kirby moves in Figure~\ref{Fig-7}, if we start from a $(-1)$-framed $2$-handle  located in the right-handed side of the $0$-framed $2$-handle and we slide it from the left  to the right of the $(-1)$-framed $2$-handle, then it will give a left-handed full twist in $T_{K'}$ and its blackboard framing will be $-(\ell -2)$. So this operation also sends  $0$-framing of $K$ to $0$-framing of $K'$. 
Hence, when we get $T_{K'}$ from $T_K$, we may add a right-handed full twist or a left-handed full twist and this operation sends $0$-framing of $K$ to $0$-framing of $K'$.  
It implies that we can perform this operation in the direction of reducing the unknotting number of $K$, \emph{i.e} the unknotting number of $K'$ is less than the unknotting number of $K$. Since unknotting number is a nonnengative integer, we get the unknot only after finitely many steps.
Therefore we conclude that
\(
X_K \sharp \mathbb{CP}^2 \) is diffeomorphic to 
\( X_U \sharp \mathbb{CP}^2 \),
where $U$ is the unknot. Furthermore, when we perform a knot surgery on $X$ using the unknot $U$, the knot surgery 4-manifold $X_U$ is the same as the original $X$. Hence we are done.  \hfill \qed

\begin{proof}[Proof of Corollary~\ref{Cor-1}]
Note that the simply connected elliptic surface $E(n)$ with Euler characteristic $12$ has a monodromy factorization of the form 
\( (\alpha \beta)^{6n}\),
where $\alpha$ and $\beta$ are right-handed Dehn twists  along simple closed curves in a generic fiber $F$ which are parallel to the circle $a$ and $b$ in Figure~\ref{Fig-2} respectively.  By using the braid relation $\alpha\beta\alpha = \beta\alpha\beta$, we always have a portion of the form $\alpha\alpha$ in the monodromy factorization of $E(n)$ because $\alpha \beta \alpha \beta = \alpha \alpha \beta \alpha$.
Since $E(n)_K = E(n) \sharp_{F = m_K \times S^1} (M_K \times S^1)$ and $F$ is a generic elliptic fiber of $E(n)$ which is the core torus in a double node neighborhood $N(\alpha \alpha)$, Theorem~\ref{Thm-1} above implies that $E(n)_K\sharp \mathbb{CP}^2$ is diffeomorphic to $E(n)\sharp \mathbb{CP}^2$. Furthermore, since it is a well-known fact that 
\(E(n)\) is almost completely decomposable, i.e., 
$E(n) \sharp \mathbb{CP}^2$ is diffeomorphic to $(2n) \mathbb{CP}^2 \sharp (10n -1) \overline{\mathbb{CP}^2}$~\cite{MANDELBAUM:1980aa, Moishezon:77}, corollary follows. 
\end{proof}


\section{Examples}

In this section, we provide a global Kirby diagram for two families of knot surgery $4$-manifolds in order to explain almost complete decomposability. The first example is the Dolgachev surface $E(1)_{2,3}$, whose almost complete decomposability was already known before. We explain this fact using an argument in the proof of main theorem. Note that  the Dolgachev surface $E(1)_{2,3}$ is diffeomorphic to a knot surgery $4$-manifold $E(1)_{3_1}$ (equivalently, $E(1)_{3_1^*}$),
where $3_1$ and $3_1^*$ denote a left- and right-handed trefoil knot in $S^3$ respectively. Hence it suffices to see how $E(1)_{3_1^*}$ is untwisted after a connected sum with $\mathbb{CP}^2$.

\medskip

\noindent
\textbf{$E(1)_{3_1^*}$ case:} A Kirby diagram of $E(1)_{3_1^*}$ can be drawn as in Figure~\ref{Fig-13} (up to $3$-handles and a $4$-handle). If we take a connected sum with $\mathbb{CP}^2$, then we get a disjoint $(+1)$-framed $2$-handle on unknot. This Kirby diagram satisfies all conditions in Theorem~\ref{Thm-1} above so that we can apply all operations used in the proof of it. Therefore the tangle part is changed as in Figure~\ref{Fig-12} and a curve corresponding to the $0$-framed longitude goes to the $0$-framed longitude of a newly generated knot which is the unknot with two positive crossings and one negative crossing. The other parts coming from $E(1) - \nu(F)$ remain unchanged, so that $E(1)_{3_1^*} \sharp \mathbb{CP}^2 $ is diffeomorphic to $E(1)_U \sharp \mathbb{CP}^2$, where $U$ is the unknot. Here is a detailed proof:
Let us first consider a portion of the Kirby diagram in Figure~\ref{Fig-13} and an extra $(+1)$-framed $2$-handle coming from $\mathbb{CP}^2$ in Figure~\ref{Fig-14}. 
We slide this $(+1)$-framed $2$-handle over two $(-1)$-framed $2$-handles as in Figure~\ref{Fig-5}, so that we get Figure~\ref{Fig-15}. Now we apply the operation in Figure~\ref{Fig-7} to get Figure~\ref{Fig-16}.
We slide the $(-1)$-framed $2$-handle over a $0$-framed $2$-handle to get Figure~\ref{Fig-17}, and 
we slide again this $2$-handle over two $0$-framed $2$-handles passing through the bottom dotted $1$-handle in Figure~\ref{Fig-13} to get Figure~\ref{Fig-18} and Figure~\ref{Fig-19}. 
And then we slide $(-1)$-framed $2$-handle along a $0$-framed longitude of $K$ over this $(-1)$-framed $2$-handle twice to get Figure~\ref{Fig-20}. Finally, by using an isotopy and a reversed operation of Figure~\ref{Fig-5}, we get Figure~\ref{Fig-21} and Figure~\ref{Fig-22}. It gives a Kirby diagram of $E(1)_U\sharp \mathbb{CP}^2$.

\medskip

\begin{figure}[htbp]
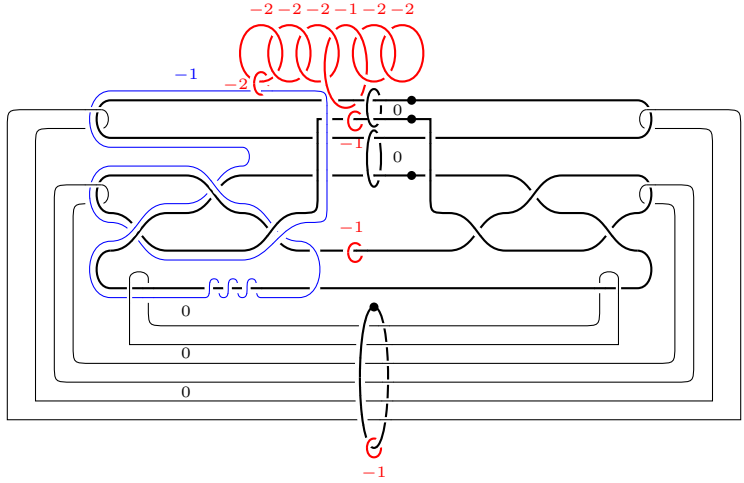

\begin{center}


\caption{A Kirby diagram of $E(1)_{3_1^*}$}
\label{Fig-13}
\end{center}
\end{figure}

\begin{figure}[htbp]
\begin{center}
%
\caption{}
\label{Fig-14}
\end{center}
\end{figure}

\begin{figure}[htbp]
\begin{center}
%

\caption{}
\label{Fig-15}
\end{center}
\end{figure}

\begin{figure}[htbp]
\begin{center}
%
\caption{}
\label{Fig-16}
\end{center}
\end{figure}

\begin{figure}[htbp]
\begin{center}
%
\caption{}
\label{Fig-22}
\end{center}
\end{figure}

\vspace*{\floatsep}

\noindent
\textbf{$E(1)_{K_n}$ case:} 
Let us consider the Stallings knot $K_n$ which is obtained from $3_1 \sharp 3_1^*$  by applying Stallings twist $n$ times. 
A Kirby diagram of $E(1)_{K_0}$ is drawn in Figure~\ref{Fig-23} (up to $3$-handles and a $4$-handle). 
By an operation explained in Figure~\ref{Fig-5}, we get a $(-1)$-framed $2$-handle along a simple closed curve $\gamma_1$ or $\gamma_2$ in Figure~\ref{Fig-23}. 
By using this $(-1)$-framed $2$-handle and by operations explained in proof of Theorem~\ref{Thm-1}, we can show that
\(E(1)_{K_0} \sharp \mathbb{CP}^2\) is diffeomorphic to  \( E(1)_{K_n} \sharp \mathbb{CP}^2 \)
for any integer $n$. Here is a sketch of proof:
For a positive integer $n$, we perform the following sequence of Kirby moves $n$ times repeatedly. 
(For a negative integer $n$, we put a $(-1)$-framed $2$-handle along $\gamma_1$ and we perform all operations from right to left direction.)

\begin{figure}[htbp]
\begin{center}
\begin{tikzpicture}[scale=0.8]
\begin{knot}[
	scale=0.35,
	clip width=5,
	clip radius = 2pt,
	end tolerance = 0.5pt,
	consider self intersections=true,
]

\strand (1,19) -- ++(22,0);
\strand (1,17) -- ++(22,0);
\strand (1,15)--++(3,0)
..controls +(1,0) and +(-1,0)..++(2,-2)..controls +(1,0) and +(-1,0)..++(2,-2)-- ++(8,0) 
..controls +(1,0) and +(-1,0)..++(2,2)..controls +(1,0) and +(-1,0)..++(2,2)--++(3,0);
\strand (1,13)--++(1,0)
..controls +(1,0) and +(-1,0)..++(2,-2)--++(2,0)..controls +(1,0) and +(-1,0)..++(2,2)-- ++(8,0) 
..controls +(1,0) and +(-1,0)..++(2,-2)--++(2,0)..controls +(1,0) and +(-1,0)..++(2,2)--++(1,0);
\strand (1,11)--++(1,0)
..controls +(1,0) and +(-1,0)..++(2,2)..controls +(1,0) and +(-1,0)..++(2,2)-- ++(12,0) 
..controls +(1,0) and +(-1,0)..++(2,-2)..controls +(1,0) and +(-1,0)..++(2,-2)--++(1,0);

\strand (1,9)--++(1,0)
..controls +(1,0) and +(-1,0)..++(2,-2)..controls +(1,0) and +(-1,0)..++(2,-2)-- ++(12,0) 
..controls +(1,0) and +(-1,0)..++(2,2)..controls +(1,0) and +(-1,0)..++(2,2)--++(1,0);
\strand (1,7) -- ++(1,0)
..controls +(1,0) and +(-1,0)..++(2,2)--++(2,0)..controls +(1,0) and +(-1,0)..++(2,-2)-- ++(8,0) 
..controls +(1,0) and +(-1,0)..++(2,2)--++(2,0)..controls +(1,0) and +(-1,0)..++(2,-2)--++(1,0);
\strand (1,5)--++(3,0)
..controls +(1,0) and +(-1,0)..++(2,2)..controls +(1,0) and +(-1,0)..++(2,2)-- ++(8,0) 
..controls +(1,0) and +(-1,0)..++(2,-2)..controls +(1,0) and +(-1,0)..++(2,-2)--++(3,0);
\strand (1,3) -- ++(22,0);
\strand (1,1) -- ++(22,0);

\strand (1,1)..controls +(-1,0) and +(-1,0)..++(0,2);
\strand (23,1)..controls +(1,0) and +(1,0)..++(0,2);

\strand (1,5)..controls +(-1,0) and +(-1,0)..++(0,2);
\strand (23,5)..controls +(1,0) and +(1,0)..++(0,2);

\strand (1,9)..controls +(-1,0) and +(-1,0)..++(0,2);
\strand (23,9)..controls +(1,0) and +(1,0)..++(0,2);

\strand (1,13)..controls +(-1,0) and +(-1,0)..++(0,2);
\strand (23,13)..controls +(1,0) and +(1,0)..++(0,2);

\strand (1,17)..controls +(-1,0) and +(-1,0)..++(0,2);
\strand (23,17)..controls +(1,0) and +(1,0)..++(0,2);

\strand[thin, blue] (1,20) -- ++(7,0);
\strand[thin, blue] (1,16.5) -- ++(6,0);
\strand[thin, blue] (1,15.5)--++(3,0)
..controls +(1.2,0) and +(-0.8,0)..++(2,-2)..controls +(1.2,0) and +(-0.8,0)..++(2,-2);
\strand[thin, blue] (1,12.5)--++(1,0)
..controls +(.8,0) and +(-1.2,0)..++(2,-2)--++(2,0)..controls +(1.2,0) and +(-.8,0)..++(2,2);
\strand[thin, blue] (1,11.5)--++(1,0)
..controls +(.8,0) and +(-1.2,0)..++(2,2)..controls +(.8,0) and +(-1.2,0)..++(2,2)-- ++(1,0);

\strand[thin, blue] (1,8.5)--++(1,0)
..controls +(.8,0) and +(-1.2,0)..++(2,-2)..controls +(.8,0) and +(-1.2,0)..++(2,-2)-- ++(1,0) ;
\strand[thin, blue] (1,7.5) -- ++(1,0)
..controls +(.8,0) and +(-1.2,0)..++(2,2)--++(2,0)..controls +(1.2,0) and +(-.8,0)..++(2,-2);
\strand[thin, blue] (1,4.5)--++(3,0)
..controls +(1.2,0) and +(-.8,0)..++(2,2)..controls +(1.2,0) and +(-.8,0)..++(2,2);
\strand[thin, blue] (1,3.5) -- ++(6,0);
\strand[thin, blue] (1,0.5) -- ++(7,0);

\strand[thin, blue] (1,0.5)..controls +(-1.5,0) and +(-1.5,0)..++(0,3);
\strand[thin, blue] (1,4.5)..controls +(-1.5,0) and +(-1.5,0)..++(0,3);
\strand[thin, blue] (1,8.5)..controls +(-1.5,0) and +(-1.5,0)..++(0,3);
\strand[thin, blue] (1,12.5)..controls +(-1.5,0) and +(-1.5,0)..++(0,3);
\strand[thin, blue] (1,16.5)..controls +(-1.5,0) and +(-1.5,0)..++(0,3.5);

\strand[thin, blue] (7,3.5)..controls +(0.5,0) and +(0.5,0)..++(0,1);
\strand[thin, blue] (7,15.5)..controls +(0.5,0) and +(0.5,0)..++(0,1);
\strand[thin, blue] (8,0.5)..controls +(0.6,0)..++(0.6,0.6)--++(0,5.8)..controls +(0,0.6)..++(-0.6,0.6);
\strand[thin, blue] (8,12.5)..controls +(0.6,0)..++(0.6,0.6)--++(0,6.3)..controls +(0,0.6)..++(-0.6,0.6);
\strand[thin, blue] (8,8.5)..controls +(0.6,0)..++(0.6,0.6)--++(0,1.8)..controls +(0,0.6)..++(-0.6,0.6);

\draw (4,20 ) node[above, blue] {\tiny $-1$};

\strand[thick, red] (16, 19.4)..controls +(0.5,0) and +(0.5,0)..++(0,-1)..controls +(-0.5,0) and +(-0.5,0)..++(0,1);
\strand[thick, red] (18, 19.4)..controls +(0.5,0) and +(0.5,0)..++(0,-1)..controls +(-0.5,0) and +(-0.5,0)..++(0,1);

\strand[thick, red] (7, 23.5)..controls +(1.5,0) and +(1.5,0) ..++(0,-3).. controls +(-1.5,0) and +(-1.5,0)..++(0,3);
\strand[thick, red] (8.5, 23.5)..controls +(1.5,0) and +(1.5,0) ..++(0,-3).. controls +(-1.5,0) and +(-1.5,0)..++(0,3);
\strand[thick, red] (10, 23.5)..controls +(1.5,0) and +(1.5,0) ..++(0,-3).. controls +(-1.5,0) and +(-1.5,0)..++(0,3);
\strand[thick, red] (11.5, 23.5)..controls +(1.5,0) and +(1.5,0) ..++(0,-5).. controls +(-1.5,0) and +(-1.5,0)..++(0,5);
\strand[thick, red] (13, 23.5)..controls +(1.5,0) and +(1.5,0) ..++(0,-3).. controls +(-1.5,0) and +(-1.5,0)..++(0,3);
\strand[thick, red] (14.5, 23.5)..controls +(1.5,0) and +(1.5,0) ..++(0,-3).. controls +(-1.5,0) and +(-1.5,0)..++(0,3);

\strand[thick, red] (7, 20.7)..controls +(0.5,0) and +(0.5,0) ..++(0,-1).. controls +(-0.5,0) and +(-0.5,0)..++(0,1);
\strand[thick, red] (13.4, -5.6)..controls +(0.5,0) and +(0.5,0) ..++(0,-1).. controls +(-0.5,0) and +(-0.5,0)..++(0,1);

\draw (16,19.4 ) node[above, red] {\tiny $-1$}; 
\draw (18,19.4) node[above, red] {\tiny $-1$}; 
\draw (7, 23.5) node[above, red] {\tiny $-2$}; 
\draw (8.5, 23.5) node[above, red] {\tiny $-2$}; 
\draw (10, 23.5) node[above, red] {\tiny $-2$}; 
\draw (11.5, 23.5) node[above, red] {\tiny $-1$}; 
\draw (13, 23.5) node[above, red] {\tiny $-2$}; 
\draw (14.5, 23.5) node[above, red] {\tiny $-2$}; 
\draw (7, 20.2) node[below left, red] {\tiny $-2$}; 
\draw (13.5, -6) node[right, red] {\tiny $-1$};

\strand[thin] (0,0)--
++(0, 1)..controls +(0,0.5)..++(0.5,0.5)--
++(1,0)..controls+(0.5,0) and +(0.5,0)..++(0,1)--
++(-1.5,0)..controls +(-0.5,0)..++(-0.5,-0.5)--
++(0,-2.5)..controls +(0,-0.5)..++(0.5,-0.5)--
++(24,0)..controls +(0.5,0)..++(0.5,0.5)--
++(0,2.5)..controls +(0,0.5)..++(-0.5,0.5)--
++(-1.5,0)..controls +(-0.5,0) and +(-0.5,0) ..++(0,-1)--
++(1,0)..controls +(0.5,0)..++(0.5,-0.5)--
++(0,-1) ..controls +(0,-0.5)..++(-0.5,-0.5)--
++(-23,0)..controls +(-0.5,0)..++(-0.5,0.5);
\strand[thin] (-1,-1)--
++(0, 6)..controls +(0,0.5)..++(0.5,0.5)--
++(1,0)..controls+(0.5,0) and +(0.5,0)..++(0,1)--
++(-1.5,0)..controls +(-0.5,0)..++(-0.5,-0.5)--
++(0,-7.5)..controls +(0,-0.5)..++(0.5,-0.5)--
++(26,0)..controls +(0.5,0)..++(0.5,0.5)--
++(0,7.5)..controls +(0,0.5)..++(-0.5,0.5)--
++(-1.5,0)..controls +(-0.5,0) and +(-0.5,0) ..++(0,-1)--
++(1,0)..controls +(0.5,0)..++(0.5,-0.5)--
++(0,-6) ..controls +(0,-0.5)..++(-0.5,-0.5)--
++(-25,0)..controls +(-0.5,0)..++(-0.5,0.5);
\strand[thin] (-2,-2)--
++(0, 11)..controls +(0,0.5)..++(0.5,0.5)--
++(2,0)..controls+(0.5,0) and +(0.5,0)..++(0,1)--
++(-2.5,0)..controls +(-0.5,0)..++(-0.5,-0.5)--
++(0,-12.5)..controls +(0,-0.5)..++(0.5,-0.5)--
++(28,0)..controls +(0.5,0)..++(0.5,0.5)--
++(0,12.5)..controls +(0,0.5)..++(-0.5,0.5)--
++(-2.5,0)..controls +(-0.5,0) and +(-0.5,0) ..++(0,-1)--
++(2,0)..controls +(0.5,0)..++(0.5,-0.5)--
++(0,-11) ..controls +(0,-0.5)..++(-0.5,-0.5)--
++(-27,0)..controls +(-0.5,0)..++(-0.5,0.5);
\strand[thin] (-3,-3)--
++(0, 16)..controls +(0,0.5)..++(0.5,0.5)--
++(3,0)..controls+(0.5,0) and +(0.5,0)..++(0,1)--
++(-3.5,0)..controls +(-0.5,0)..++(-0.5,-0.5)--
++(0,-17.5)..controls +(0,-0.5)..++(0.5,-0.5)--
++(30,0)..controls +(0.5,0)..++(0.5,0.5)--
++(0,17.5)..controls +(0,0.5)..++(-0.5,0.5)--
++(-3.5,0)..controls +(-0.5,0) and +(-0.5,0) ..++(0,-1)--
++(3,0)..controls +(0.5,0)..++(0.5,-0.5)--
++(0,-16) ..controls +(0,-0.5)..++(-0.5,-0.5)--
++(-29,0)..controls +(-0.5,0)..++(-0.5,0.5);
\strand[thin] (-4,-4)--
++(0, 21)..controls +(0,0.5)..++(0.5,0.5)--
++(4,0)..controls+(0.5,0) and +(0.5,0)..++(0,1)--
++(-4.5,0)..controls +(-0.5,0)..++(-0.5,-0.5)--
++(0,-22.5)..controls +(0,-0.5)..++(0.5,-0.5)--
++(32,0)..controls +(0.5,0)..++(0.5,0.5)--
++(0,22.5)..controls +(0,0.5)..++(-0.5,0.5)--
++(-4.5,0)..controls +(-0.5,0) and +(-0.5,0) ..++(0,-1)--
++(4,0)..controls +(0.5,0)..++(0.5,-0.5)--
++(0,-21) ..controls +(0,-0.5)..++(-0.5,-0.5)--
++(-31,0)..controls +(-0.5,0)..++(-0.5,0.5);

\strand (13, 19.6)..controls +(0.5,0) and +(0.5,0)..++(0,-7)..controls +(-0.5,0) and +(-0.5,0)..++(0,7);
\strand (13, 0.15)..controls +(0.5,0) and +(0.5,0)..++(0,-6.5)..controls +(-0.5,0) and +(-0.5,0)..++(0,6.5);
\strand (10, 17.6)..controls +(0.5,0) and +(0.5,0)..++(0,-3)..controls +(-0.5,0) and +(-0.5,0)..++(0,3);
\strand (10, 5.6)..controls +(0.5,0) and +(0.5,0)..++(0,-3)..controls +(-0.5,0) and +(-0.5,0)..++(0,3);
\strand (10, 11.6)..controls +(0.5,0) and +(0.5,0)..++(0,-3)..controls +(-0.5,0) and +(-0.5,0)..++(0,3);
\strand[brown] (14, 13.6)..controls +(0.5,0) and +(0.5,0)..++(0,-7)..controls +(-0.5,0) and +(-0.5,0)..++(0,7);
\strand[brown] (12, 13.6)..controls +(0.5,0) and +(0.5,0)..++(0,-7)..controls +(-0.5,0) and +(-0.5,0)..++(0,7);

\draw (13.5, 16) node[right] {\tiny $0$};
\draw (11.5, 10) node[left] {\tiny $0$};
\draw (11.5, 16) node[left] {\tiny $0$};
\draw (11.5, 4) node[left] {\tiny $0$};
\draw[brown] (14.2,10) node[right] {\tiny $\gamma_1$};
\draw[brown] (12,10) node[right] {\tiny $\gamma_2$};

 \draw (15, 19) node[circle, fill, inner sep=1.2pt, black] {}; 
\draw (15, 15) node[circle, fill, inner sep=1.2pt, black] {}; 
\draw (15, 11) node[circle, fill, inner sep=1.2pt, black] {}; 
\draw (15, 9) node[circle, fill, inner sep=1.2pt, black] {}; 
\draw (15, 3) node[circle, fill, inner sep=1.2pt, black] {}; 
\draw (13, 0.15) node[circle, fill, inner sep=1.2pt, black] {};

\flipcrossings {78, 82, 86, 90, 94, 106, 108, 110, 112, 114};
\flipcrossings {40, 101, 28, 29, 17, 18, 100, 55, 65, 48, 49, 103};
\flipcrossings {128, 130, 132, 134, 136, 138, 140, 142, 144, 146, 148};
\flipcrossings {14, 45, 21, 68, 52, 75, 9, 12, 43, 11, 42, 33, 37, 25, 72, 62, 24, 71, 35, 23, 70,22, 69, 60};
\flipcrossings {98, 118, 116, 120, 122, 124, 126,7};
\flipcrossings {1, 10, 41, 31, 19, 66, 58, 50, 73, 76, 26, 38, 53, 80, 84, 88, 92, 96, 3, 5};
\end{knot}
\useasboundingbox ([shift={(1mm,1mm)}]current bounding box.north east) rectangle ([shift={(-1mm,-1mm)}]current bounding box.south west);
\end{tikzpicture}

\caption{A Kirby diagram of $E(1)_{K_0}$ }
\label{Fig-23}
\end{center}
\end{figure}

\begin{itemize}
\item First we locate two $(-1)$-framed $2$-handles on the fourth and seventh strands (counted from the top) and then we perform  an operation in Figure~\ref{Fig-5}, so that we get a $(-1)$-framed $2$-handle along $\gamma_2$.
\item We perform an operation in Figure~\ref{Fig-7}. Then this $2$-handle will be located at $\gamma_1$ and the fourth and seventh strands will be twisted correspondingly.
\item We slide this $(-1)$-framed $2$-handle over three $0$-framed $2$-handles located at the top, middle and bottom left sides and then we slide it over two $0$-framed $2$-handles passing through the bottom dotted $1$-handle which are located at the first and fifth counted from the bottom. We slide it again over the three $0$-framed $2$-handles located at the top, middle and bottom left sides.
\item We slide a $(-1)$-framed $2$-handle on the $0$-framed longitude of $K_0$ over this $(-1)$-framed $2$-handle twice. Then the $2$-handle goes back to its original position $\gamma_2$ with a framing $(-1)$,  the knot $K_0$ is changed to $K_1$ and a $0$-framed longitude of $K_0$ becomes a $0$-framed longitude of $K_1$.
\item We perform an operation in Figure~\ref{Fig-5}  reversely to get a disjoint  $(+1)$-framed $2$-handle and this process sends back two vanishing cycles to its original position.
\end{itemize}

Finally we can show that \(E(1)_{K_0} \sharp \mathbb{CP}^2\) is diffeomorphic to  \( E(1) \sharp \mathbb{CP}^2 \)
by using the same argument as in $E(1)_{3_1^*}$ case with Figure~\ref{Fig-24}.

\begin{figure}[htbp]
\begin{center}
\begin{tikzpicture}[scale=0.8]
\begin{knot}[
	scale=0.35,
	clip width=5,
	clip radius = 2pt,
	end tolerance = 0.5pt,
	consider self intersections=true,
]

\strand (1,19) -- ++(22,0);
\strand (1,17) -- ++(22,0);
\strand (1,15)--++(3,0)
..controls +(1,0) and +(-1,0)..++(2,-2)..controls +(1,0) and +(-1,0)..++(2,-2)-- ++(8,0) 
..controls +(1,0) and +(-1,0)..++(2,2)..controls +(1,0) and +(-1,0)..++(2,2)--++(3,0);
\strand (1,13)--++(1,0)
..controls +(1,0) and +(-1,0)..++(2,-2)--++(2,0)..controls +(1,0) and +(-1,0)..++(2,2)-- ++(8,0) 
..controls +(1,0) and +(-1,0)..++(2,-2)--++(2,0)..controls +(1,0) and +(-1,0)..++(2,2)--++(1,0);
\strand (1,11)--++(1,0)
..controls +(1,0) and +(-1,0)..++(2,2)..controls +(1,0) and +(-1,0)..++(2,2)-- ++(12,0) 
..controls +(1,0) and +(-1,0)..++(2,-2)..controls +(1,0) and +(-1,0)..++(2,-2)--++(1,0);

\strand (1,9)--++(1,0)
..controls +(1,0) and +(-1,0)..++(2,-2)..controls +(1,0) and +(-1,0)..++(2,-2)-- ++(12,0) 
..controls +(1,0) and +(-1,0)..++(2,2)..controls +(1,0) and +(-1,0)..++(2,2)--++(1,0);
\strand (1,7) -- ++(1,0)
..controls +(1,0) and +(-1,0)..++(2,2)--++(2,0)..controls +(1,0) and +(-1,0)..++(2,-2)-- ++(8,0) 
..controls +(1,0) and +(-1,0)..++(2,2)--++(2,0)..controls +(1,0) and +(-1,0)..++(2,-2)--++(1,0);
\strand (1,5)--++(3,0)
..controls +(1,0) and +(-1,0)..++(2,2)..controls +(1,0) and +(-1,0)..++(2,2)-- ++(8,0) 
..controls +(1,0) and +(-1,0)..++(2,-2)..controls +(1,0) and +(-1,0)..++(2,-2)--++(3,0);
\strand (1,3) -- ++(22,0);
\strand (1,1) -- ++(22,0);

\strand (1,1)..controls +(-1,0) and +(-1,0)..++(0,2);
\strand (23,1)..controls +(1,0) and +(1,0)..++(0,2);

\strand (1,5)..controls +(-1,0) and +(-1,0)..++(0,2);
\strand (23,5)..controls +(1,0) and +(1,0)..++(0,2);

\strand (1,9)..controls +(-1,0) and +(-1,0)..++(0,2);
\strand (23,9)..controls +(1,0) and +(1,0)..++(0,2);

\strand (1,13)..controls +(-1,0) and +(-1,0)..++(0,2);
\strand (23,13)..controls +(1,0) and +(1,0)..++(0,2);

\strand (1,17)..controls +(-1,0) and +(-1,0)..++(0,2);
\strand (23,17)..controls +(1,0) and +(1,0)..++(0,2);

\strand[thin, blue] (1,20) -- ++(7,0);
\strand[thin, blue] (1,16.5) -- ++(6,0);
\strand[thin, blue] (1,15.5)--++(3,0)
..controls +(1.2,0) and +(-0.8,0)..++(2,-2)..controls +(1.2,0) and +(-0.8,0)..++(2,-2);
\strand[thin, blue] (1,12.5)--++(1,0)
..controls +(.8,0) and +(-1.2,0)..++(2,-2)--++(2,0)..controls +(1.2,0) and +(-.8,0)..++(2,2);
\strand[thin, blue] (1,11.5)--++(1,0)
..controls +(.8,0) and +(-1.2,0)..++(2,2)..controls +(.8,0) and +(-1.2,0)..++(2,2)-- ++(1,0);

\strand[thin, blue] (1,8.5)--++(1,0)
..controls +(.8,0) and +(-1.2,0)..++(2,-2)..controls +(.8,0) and +(-1.2,0)..++(2,-2)-- ++(1,0) ;
\strand[thin, blue] (1,7.5) -- ++(1,0)
..controls +(.8,0) and +(-1.2,0)..++(2,2)--++(2,0)..controls +(1.2,0) and +(-.8,0)..++(2,-2);
\strand[thin, blue] (1,4.5)--++(3,0)
..controls +(1.2,0) and +(-.8,0)..++(2,2)..controls +(1.2,0) and +(-.8,0)..++(2,2);
\strand[thin, blue] (1,3.5) -- ++(6,0);
\strand[thin, blue] (1,0.5) -- ++(7,0);

\strand[thin, blue] (1,0.5)..controls +(-1.5,0) and +(-1.5,0)..++(0,3);
\strand[thin, blue] (1,4.5)..controls +(-1.5,0) and +(-1.5,0)..++(0,3);
\strand[thin, blue] (1,8.5)..controls +(-1.5,0) and +(-1.5,0)..++(0,3);
\strand[thin, blue] (1,12.5)..controls +(-1.5,0) and +(-1.5,0)..++(0,3);
\strand[thin, blue] (1,16.5)..controls +(-1.5,0) and +(-1.5,0)..++(0,3.5);

\strand[thin, blue] (7,3.5)..controls +(0.5,0) and +(0.5,0)..++(0,1);
\strand[thin, blue] (7,15.5)..controls +(0.5,0) and +(0.5,0)..++(0,1);
\strand[thin, blue] (8,0.5)..controls +(0.6,0)..++(0.6,0.6)--++(0,5.8)..controls +(0,0.6)..++(-0.6,0.6);
\strand[thin, blue] (8,12.5)..controls +(0.6,0)..++(0.6,0.6)--++(0,6.3)..controls +(0,0.6)..++(-0.6,0.6);
\strand[thin, blue] (8,8.5)..controls +(0.6,0)..++(0.6,0.6)--++(0,1.8)..controls +(0,0.6)..++(-0.6,0.6);

\draw (4,20 ) node[above, blue] {\tiny $-1$};

\strand[thick, red] (16, 19.4)..controls +(0.5,0) and +(0.5,0)..++(0,-1)..controls +(-0.5,0) and +(-0.5,0)..++(0,1);
\strand[thick, red] (18, 19.4)..controls +(0.5,0) and +(0.5,0)..++(0,-1)..controls +(-0.5,0) and +(-0.5,0)..++(0,1);

\strand[thick, red] (7, 23.5)..controls +(1.5,0) and +(1.5,0) ..++(0,-3).. controls +(-1.5,0) and +(-1.5,0)..++(0,3);
\strand[thick, red] (8.5, 23.5)..controls +(1.5,0) and +(1.5,0) ..++(0,-3).. controls +(-1.5,0) and +(-1.5,0)..++(0,3);
\strand[thick, red] (10, 23.5)..controls +(1.5,0) and +(1.5,0) ..++(0,-3).. controls +(-1.5,0) and +(-1.5,0)..++(0,3);
\strand[thick, red] (11.5, 23.5)..controls +(1.5,0) and +(1.5,0) ..++(0,-5).. controls +(-1.5,0) and +(-1.5,0)..++(0,5);
\strand[thick, red] (13, 23.5)..controls +(1.5,0) and +(1.5,0) ..++(0,-3).. controls +(-1.5,0) and +(-1.5,0)..++(0,3);
\strand[thick, red] (14.5, 23.5)..controls +(1.5,0) and +(1.5,0) ..++(0,-3).. controls +(-1.5,0) and +(-1.5,0)..++(0,3);

\strand[thick, red] (7, 20.7)..controls +(0.5,0) and +(0.5,0) ..++(0,-1).. controls +(-0.5,0) and +(-0.5,0)..++(0,1);
\strand[thick, red] (13.4, -5.6)..controls +(0.5,0) and +(0.5,0) ..++(0,-1).. controls +(-0.5,0) and +(-0.5,0)..++(0,1);

\draw (16,19.4 ) node[above, red] {\tiny $-1$}; 
\draw (18,19.4) node[above, red] {\tiny $-1$}; 
\draw (7, 23.5) node[above, red] {\tiny $-2$}; 
\draw (8.5, 23.5) node[above, red] {\tiny $-2$}; 
\draw (10, 23.5) node[above, red] {\tiny $-2$}; 
\draw (11.5, 23.5) node[above, red] {\tiny $-1$}; 
\draw (13, 23.5) node[above, red] {\tiny $-2$}; 
\draw (14.5, 23.5) node[above, red] {\tiny $-2$}; 
\draw (7, 20.2) node[below left, red] {\tiny $-2$}; 
\draw (13.5, -6) node[right, red] {\tiny $-1$};

\strand[thin] (0,0)--
++(0, 1)..controls +(0,0.5)..++(0.5,0.5)--
++(1,0)..controls+(0.5,0) and +(0.5,0)..++(0,1)--
++(-1.5,0)..controls +(-0.5,0)..++(-0.5,-0.5)--
++(0,-2.5)..controls +(0,-0.5)..++(0.5,-0.5)--
++(24,0)..controls +(0.5,0)..++(0.5,0.5)--
++(0,2.5)..controls +(0,0.5)..++(-0.5,0.5)--
++(-1.5,0)..controls +(-0.5,0) and +(-0.5,0) ..++(0,-1)--
++(1,0)..controls +(0.5,0)..++(0.5,-0.5)--
++(0,-1) ..controls +(0,-0.5)..++(-0.5,-0.5)--
++(-23,0)..controls +(-0.5,0)..++(-0.5,0.5);
\strand[thin] (-1,-1)--
++(0, 6)..controls +(0,0.5)..++(0.5,0.5)--
++(1,0)..controls+(0.5,0) and +(0.5,0)..++(0,1)--
++(-1.5,0)..controls +(-0.5,0)..++(-0.5,-0.5)--
++(0,-7.5)..controls +(0,-0.5)..++(0.5,-0.5)--
++(26,0)..controls +(0.5,0)..++(0.5,0.5)--
++(0,7.5)..controls +(0,0.5)..++(-0.5,0.5)--
++(-1.5,0)..controls +(-0.5,0) and +(-0.5,0) ..++(0,-1)--
++(1,0)..controls +(0.5,0)..++(0.5,-0.5)--
++(0,-6) ..controls +(0,-0.5)..++(-0.5,-0.5)--
++(-25,0)..controls +(-0.5,0)..++(-0.5,0.5);
\strand[thin] (-2,-2)--
++(0, 11)..controls +(0,0.5)..++(0.5,0.5)--
++(2,0)..controls+(0.5,0) and +(0.5,0)..++(0,1)--
++(-2.5,0)..controls +(-0.5,0)..++(-0.5,-0.5)--
++(0,-12.5)..controls +(0,-0.5)..++(0.5,-0.5)--
++(28,0)..controls +(0.5,0)..++(0.5,0.5)--
++(0,12.5)..controls +(0,0.5)..++(-0.5,0.5)--
++(-2.5,0)..controls +(-0.5,0) and +(-0.5,0) ..++(0,-1)--
++(2,0)..controls +(0.5,0)..++(0.5,-0.5)--
++(0,-11) ..controls +(0,-0.5)..++(-0.5,-0.5)--
++(-27,0)..controls +(-0.5,0)..++(-0.5,0.5);
\strand[thin] (-3,-3)--
++(0, 16)..controls +(0,0.5)..++(0.5,0.5)--
++(3,0)..controls+(0.5,0) and +(0.5,0)..++(0,1)--
++(-3.5,0)..controls +(-0.5,0)..++(-0.5,-0.5)--
++(0,-17.5)..controls +(0,-0.5)..++(0.5,-0.5)--
++(30,0)..controls +(0.5,0)..++(0.5,0.5)--
++(0,17.5)..controls +(0,0.5)..++(-0.5,0.5)--
++(-3.5,0)..controls +(-0.5,0) and +(-0.5,0) ..++(0,-1)--
++(3,0)..controls +(0.5,0)..++(0.5,-0.5)--
++(0,-16) ..controls +(0,-0.5)..++(-0.5,-0.5)--
++(-29,0)..controls +(-0.5,0)..++(-0.5,0.5);
\strand[thin] (-4,-4)--
++(0, 21)..controls +(0,0.5)..++(0.5,0.5)--
++(4,0)..controls+(0.5,0) and +(0.5,0)..++(0,1)--
++(-4.5,0)..controls +(-0.5,0)..++(-0.5,-0.5)--
++(0,-22.5)..controls +(0,-0.5)..++(0.5,-0.5)--
++(32,0)..controls +(0.5,0)..++(0.5,0.5)--
++(0,22.5)..controls +(0,0.5)..++(-0.5,0.5)--
++(-4.5,0)..controls +(-0.5,0) and +(-0.5,0) ..++(0,-1)--
++(4,0)..controls +(0.5,0)..++(0.5,-0.5)--
++(0,-21) ..controls +(0,-0.5)..++(-0.5,-0.5)--
++(-31,0)..controls +(-0.5,0)..++(-0.5,0.5);

\strand (13, 19.6)..controls +(0.5,0) and +(0.5,0)..++(0,-7)..controls +(-0.5,0) and +(-0.5,0)..++(0,7);
\strand (13, 7.6)..controls +(0.5,0) and +(0.5,0)..++(0,-7)..controls +(-0.5,0) and +(-0.5,0)..++(0,7);
\strand (13, 0.15)..controls +(0.5,0) and +(0.5,0)..++(0,-6.5)..controls +(-0.5,0) and +(-0.5,0)..++(0,6.5);
\strand (12, 17.6)..controls +(0.5,0) and +(0.5,0)..++(0,-3)..controls +(-0.5,0) and +(-0.5,0)..++(0,3);
\strand (12, 5.6)..controls +(0.5,0) and +(0.5,0)..++(0,-3)..controls +(-0.5,0) and +(-0.5,0)..++(0,3);
\strand[brown] (10, 13.6)..controls +(0.5,0) and +(0.5,0)..++(0,-3)..controls +(-0.5,0) and +(-0.5,0)..++(0,3);
\strand[brown] (16, 9.6)..controls +(0.5,0) and +(0.5,0)..++(0,-3)..controls +(-0.5,0) and +(-0.5,0)..++(0,3);

\draw (13.5, 16) node[right] {\tiny $0$};
\draw (13.5, 4) node[right] {\tiny $0$};
\draw (11.5, 16) node[left] {\tiny $0$};
\draw (11.5, 4) node[left] {\tiny $0$};
\draw[brown] (10,13.6) node[above] {\tiny $\gamma_3$};
\draw[brown] (16,6.6) node[below] {\tiny $\gamma_4$};

 \draw (15, 19) node[circle, fill, inner sep=1.2pt, black] {}; 
\draw (15, 15) node[circle, fill, inner sep=1.2pt, black] {}; 
\draw (15, 11) node[circle, fill, inner sep=1.2pt, black] {}; 
\draw (15, 9) node[circle, fill, inner sep=1.2pt, black] {}; 
\draw (15, 3) node[circle, fill, inner sep=1.2pt, black] {}; 
\draw (13, 0.15) node[circle, fill, inner sep=1.2pt, black] {};

\flipcrossings {1,10,35,27, 19, 62, 54, 44, 65, 70, 14, 39, 9, 12, 37, 29, 11, 36, 31, 21, 48, 69, 56, 46, 67, 72, 45, 66, 126, 128, 130, 132, 134, 136, 138, 140, 142, 144, 124};
\flipcrossings {74, 78, 82, 86, 90, 102, 104, 106, 108, 110};
\flipcrossings {34, 24, 25, 97, 17, 18, 96, 51, 61, 42,43,99};
\flipcrossings {94, 114, 112, 116, 118, 120, 122, 7, 3, 5};
\flipcrossings {76, 80, 84, 88, 92, 22, 32, 49};
\flipcrossings {57, 63};
\end{knot}
\useasboundingbox ([shift={(1mm,1mm)}]current bounding box.north east) rectangle ([shift={(-1mm,-1mm)}]current bounding box.south west);
\end{tikzpicture}

\caption{Another Kirby diagram of $E(1)_{K_0}$ }
\label{Fig-24}
\end{center}
\end{figure}

\begin{rem}
As mentioned in Section 2, R. Gompf~\cite{Gompf:2016a, Gompf:2016} constructed an infinite order cork and
M. Tange~\cite{Tange:2016} also constructed an example of   infinite order corks, a $\mathbb{Z}^k$-cork.
Note that Gompf's $\mathbb{Z}$-cork is related to a knot surgery $4$-manifold $E(n)_{K_k}$, where $K_k$ is the twisted knot and Tange's $\mathbb{Z}^k$-cork is related to $E(k)_{K(n_1, n_2, \cdots, n_k)}$, where $K(n_1, n_2, \cdots, n_k)=K_{1, n_1} \sharp K_{2, n_2} \sharp \cdots \sharp K_{k, n_k}$ and $K_{m,n}$ is a $2$-bridge knot of type $C(2m+1, -2n, 2)$.
So all these examples can also be dissolved under a connected sum with $\mathbb{CP}^2$. 
\end{rem}

\begin{rem}
Although it is known that every elliptic Lefschetz fibration is almost completely decomposable~\cite{Mandelbaum:1985ew}, it is not sure whether every simply connected symplectic Lefschetz fibration over $S^2$ with a high genus fiber is almost completely decomposable or not. Nevertheless, since a family of knot surgery $4$-manifolds $E(n)_K$ with a fibered knot $K$ admit a symplectic Lefschetz fibration over $S^2$ with a high genus fiber~\cite{FS:2004} and they are almost completely decomposable, it is an intriguing question whether all simply connected symplectic Lefschetz fibration over $S^2$ are almost completely decomposable or not.
\end{rem}


\providecommand{\bysame}{\leavevmode\hbox to3em{\hrulefill}\thinspace}

\end{document}